\documentclass[a4paper, 12pt]{amsart}
\usepackage{amsmath}
\usepackage{amsthm}
\usepackage{amssymb}
\usepackage{framed}
\usepackage[hidelinks]{hyperref}
\usepackage{xcolor}
\usepackage{graphicx}
\usepackage[a4paper,  margin=2.5cm]{geometry}

\newtheorem{lm}{Lemma}

\newtheorem{te}{Theorem}
\newtheorem{re}{Remark}
\theoremstyle{definition}

\newcommand{\dif}{\mathrm{d}}

\newcommand{\norm}[1]{\Vert #1 \Vert}
\newcommand{\abs}[1]{\vert #1 \vert}
\newcommand{\ov}{\overline}

\newcommand{\RR}{\mathbb{R}}

\newcommand{\NN}{\mathbb{N}}

\DeclareMathOperator{\di}{div}

\begin{document}
\title[Liouville for primitive equations]{A Liouville-type theorem for the 3D primitive~equations}

\author{D. Peralta-Salas}
\address{Instituto de Ciencias Matem\'aticas, Consejo Superior de
 Investigaciones Cient\'\i ficas, 28049 Madrid, Spain}
\email{dperalta@icmat.es}

\author{R. Slobodeanu}
\address{Faculty of Physics, University of Bucharest, P.O. Box Mg-11, RO--077125 Bucharest-M\u agurele, Romania}
\email{radualexandru.slobodeanu@g.unibuc.ro}

\date{\today}

\subjclass[2020]{35Q35, 86A10.}

\keywords{Primitive equations, stationary solutions, compact support.}

\maketitle


\begin{abstract}
The 3D primitive equations are used in most geophysical fluid models to approximate the large scale oceanic and atmospheric dynamics. We prove that there do not exist smooth stationary solutions to the 3D primitive equations with compact support, independently of the presence of the Coriolis rotation term or the viscosity. This result is in strong contrast with the recently established existence of compactly supported smooth solutions to the incompressible 3D Euler equations.
\end{abstract}

\maketitle

\section{Introduction}\label{S:intro}
The classical Liouville's theorem in Complex Analysis states that every bounded entire function must be constant. Analogously any bounded harmonic function defined on $\RR^n$ must be constant. By extension, a Liouville-type theorem for a PDE is any result that asserts the triviality of (smooth enough) stationary solutions defined on the whole space, which have constant behaviour at infinity (or other boundedness properties). For fluid and fluid related equations this type of rigidity has been studied extensively and Liouville-type results are known for Navier-Stokes equations (with axisymmetry)~\cite{koch, chae2}, forced Navier-Stokes and Euler equations~\cite{chae1}, MHD and Hall-MHD systems~\cite{chae3, LiNiu},  (generalized) Beltrami equations~\cite{N,CC}, as well as for the Euler equations with axisymmetry and no swirl~\cite{JX}, to mention only a few from a plethora of articles on this topic. In most cases, a Liouville-type result bans the existence of compactly supported solutions of the respective equations. From this perspective (but also for other reasons ~\cite{CV}) the recent construction by Gavrilov~\cite{Gavrilov} of a $C^\infty$ compactly supported solution to the incompressible Euler equations in $\mathbb R^3$ is a remarkable fact (see also~\cite{CV},~\cite{DEP} and~\cite{rove}).

Motivated by this surprising result, in this paper we focus on a different system of equations, from Geophysical Fluid Dynamics this time, namely the primitive equations. The 3D primitive equations are used to model the large scale oceanic and atmospheric dynamics and they are at the core of the numerical general circulation models (GCMs) used for weather and climate prediction~\cite{tem}. They can be seen as the small aspect ratio limit of the Navier-Stokes equations~\cite{LiTi}, with rotation term and coupled to thermodynamics. Introduced a century ago by Bjerkness (later simplified by Richardson), the primitive equations have been thoroughly studied in the last decades since their comeback~\cite{LTWa, LTW} in the context of an increased computational power. Many remarkable existence, uniqueness and regularity results in both viscous and inviscid cases are now available in the literature (see~\cite{tem,titi,titi2} and references therein).

The stationary (i.e., time independent) primitive equations are the following system of PDEs that the velocity field $X=ue_1+ve_2+we_3$, the hydrodynamic pressure  $p$ and the temperature $T$ of the fluid have to satisfy:
\begin{equation}\label{primitiveq}
\begin{aligned}
&X\cdot \nabla u + \partial_x p - Rv = \nu_H \Delta_H u + \nu_3 \partial_{zz}u \,,  \\
&X\cdot \nabla v + \partial_y p + Ru = \nu_H \Delta_H v + \nu_3 \partial_{zz}v\,, \\
&\partial_z p+T =0\,, \\
&X\cdot \nabla T =\kappa_{H} \Delta_{H} T+\kappa_{3} \partial_{zz}T\,, \\
&\di X =0\,,
\end{aligned}
\end{equation}
on the horizontal channel
$$\Omega=\left\{(x, y, z)\in\mathbb R^3: 0 < z < L \right\}\,,$$
with boundary conditions: no normal flow, that is
\[
w|_{z=0}=w|_{z=L}=0\,,
\]
and no heat flux in the vertical direction,
\[
\partial_z T|_{z=0}=\partial_z T|_{z=L}=0\,.
\]
Here $\kappa_H>0$ and $\kappa_3>0$ are the diffusion constants, $\nu_H \geqslant 0$ and $\nu_3 \geqslant 0$ are the viscosity coefficients, and $\Delta_H$ stands for the horizontal Laplacian $\Delta_H:=\partial_{xx}+\partial_{yy}$. The rotation parameter $R\in\mathbb R$ controls the Coriolis forcing term. The physical significance of the equations is given by the conservation of horizontal momentum (first two equations), hydrostatic balance, heat conduction and continuity equation, respectively.

The goal of this note is to study the existence of $C^\infty$ solutions to Equations~\eqref{primitiveq} such that the velocity field $X$ has compact support contained in the channel~$\Omega$. Our main result is the following Liouville-type theorem establishing that such compactly supported stationary solutions do not exist (unless trivial). This mathematical fact has a clear physical relevance: one cannot have localized (in a compact region) steady motions of the ocean or of the atmosphere that are elsewhere at rest, while such ``pathological'' objects do exist for the 3D incompressible Euler equations~\cite{Gavrilov}.

Recall that the support of a vector field is the closure of the set of points where the field does not vanish (it is then a closed set by definition).

\begin{te}\label{th.main}
Let $(X,p,T)$ be a $C^\infty$ solution to the stationary primitive equations~\eqref{primitiveq} with the aforementioned boundary conditions. Assume that the velocity field $X$ has compact support contained in $\Omega$. Then $X\equiv 0$, $p=-az+b$ and $T=a$ for some constants $a,b\in \mathbb R$.
\end{te}
\begin{re}
The same result holds true for compactly supported smooth solutions of the primitive equations defined on the whole $\RR^3$, provided that they satisfy the heat flux boundary condition
\[
\lim_{z\to\pm \infty} \partial_zT=0\,.
\]
\end{re}

\begin{re}
In general, the time dependent inviscid 3D primitive equations are ill-posed in all Sobolev spaces. This follows from the recently established fact that solutions in a neighborhood of certain steady states exhibit Kelvin-Helmholtz instability~\cite{Titi21}. This suggests that a natural space for well-posedness of the time dependent inviscid primitive equations is the space of real analytic vector fields, which is indeed the case, at least under suitable conditions~\cite{Titi2-22}. Of course, in the setting of analytic functions, no compactly supported solutions can exist.
\end{re}

The proof of Theorem~\ref{th.main} is divided in four steps. In the first one we get rid of the pressure function and reduce the problem to a set of partial differential equations that only involve the velocity field $X$. The case of positive viscosity $\nu_H+\nu_3>0$ is elementary and directly follows from integration by parts (this may also be understood as a consequence of partial analytic hypoelliptic regularity, see~\cite{Titi22} for this property in the time dependent case); it is presented in Step~4. The proof for the case $\nu_H=\nu_3=0$ and no Coriolis term, i.e., $R=0$, can be settled using elementary identities for divergence-free vector fields (cf. Step~2). The most involved case is the non viscous one with Coriolis rotation term $R\neq 0$, which is Step~3. The proof of this part is based on the existence of two nontrivial first integrals of the vector field $X$. The argument we develop is rather unusual in the analysis of PDEs, and is based on a local flow box argument combined with the classical Thom's isotopy theorem from differential topology~\cite{abra}.

We emphasize that the assumption that the support of $X$ is contained in the open set $\Omega$ is important. Indeed, there are smooth solutions to the inviscid (i.e., $\nu_H=\nu_3=0$) stationary primitive equations such that $X$ has compact support contained in $\overline{\Omega}$. When $R\neq 0$, one can take solutions of the form
\begin{equation*}
X=R\psi' \, (ye_1-xe_2)\,,\qquad p=-az+b+\frac12R^2\int_0^{x^2+y^2} \!\! \psi'(s)^2 ds-\frac12R^2\psi\,,\qquad T=a\,,
\end{equation*}
while for $R=0$ the solution is given by
\begin{equation*}
X=\psi' \, (ye_1-xe_2)\,,\qquad p=-az+b+\frac12\int_0^{x^2+y^2} \!\! \psi'(s)^2 ds\,,\qquad T=a\,.
\end{equation*}
Here $a,b\in\mathbb R$ and $\psi\equiv \psi(x^2+y^2)$ is any $C^\infty$ compactly supported function.

The proof of Theorem~\ref{th.main} is presented in Section~\ref{S:proof}. In the last section (Section~\ref{S:final}) we discuss the particular case of axisymmetric stationary solutions in comparison with Gavrilov's compactly supported steady states of the 3D Euler equations.

\section{Proof of the main theorem}\label{S:proof}

We divide the proof in four steps. In the first step we use elementary arguments to show that a compactly supported velocity field $X$ implies that the pressure and the temperature must be of the form $p=-az+b$ and $T=a$, respectively, for some constants $a,b\in\mathbb R$ on the whole $\Omega$. This allows us to reduce the problem to a PDE that involves only the vector field $X$. The second step proves the theorem in the case when the Coriolis term is $R=0$ and there is no viscosity ($\nu_H=\nu_3=0$), the third step addresses the case with rotation ($R \neq 0$) and no viscosity ($\nu_H=\nu_3=0$), while the viscous case is considered in the last step (just integration by parts). The proof of Step~2 is based on certain simple identities for divergence-free vector fields, while the proof of Step~3 is more involved and exploits the existence of two special first integrals for the vector field $X$. All along this section we shall denote the volume element $dxdydz$ by $d\mu$.

\subsection{Step~1: reduction of the equations}

Let $K\subset \Omega$ be the compact set where $X$ is supported. On the set $\Omega\backslash K$ the first and second primitive equations imply that $p\equiv p(z)$, which in turn implies that $T\equiv T(z)$ by the third primitive equation. Accordingly, we infer from the fourth primitive equation and the zero heat flux boundary condition that
\[
T=a
\]
for some real constant $a$, and hence
$$p=-az+b$$
for some constant $b$, on the complement of $K$.

Next we observe that multiplying the fourth primitive equation by $T$ and integrating on $\Omega$, we obtain the identity
\begin{equation}\label{eq.T}
\frac12\int_{\Omega}X\cdot \nabla T^2 \dif \mu =\kappa_H \int_\Omega T\Delta_H T \dif \mu + \kappa_3\int_\Omega T\partial_{zz}T \dif \mu \,,
\end{equation}
which is well defined because $T=a$ in the complement of $K$. Now we notice that the condition $\di X=0$ implies that
\[
\int_{\Omega}X\cdot \nabla T^2 \dif \mu =\int_{\partial\Omega}T^2X\cdot Nd\sigma=0\,,
\]
where we have used that $X$ is supported on $K\subset\Omega$ (here $N$ is the outer unit normal vector on $\partial\Omega$). Moreover, using that $T$ is constant in $\Omega\backslash K$, we also infer that
\[
\int_\Omega T\Delta_H T \dif \mu =-\int_{\Omega}\Big((\partial_xT)^2+(\partial_yT)^2\Big)\dif \mu \,,
\]
and
\[
\int_{\Omega}T\partial_{zz}T \dif \mu =-\int_{\Omega}(\partial_zT)^2 \dif \mu \,.
\]
It then follows from Equation~\eqref{eq.T} and the value of $p$ in $\Omega\backslash K$ that
\begin{equation*}
T=a\,, \qquad p=-az+b\,,
\end{equation*}
on the whole $\Omega$.

This immediately yields that the primitive equations~\eqref{primitiveq} reduce to a system of PDEs that involves only the vector field $X$:
\begin{equation} \label{prim}
\begin{aligned}
&X\cdot \nabla u - Rv = \nu_H \Delta_H u + \nu_3 \partial_{zz}u\,, \\
&X\cdot \nabla v + Ru = \nu_H \Delta_H v + \nu_3 \partial_{zz}v\,, \\
&\di X=0 \,.
\end{aligned}
\end{equation}

\begin{re}
The fact that a compactly supported solution $X$ to the 3D primitive equations also satisfies the reduced equations~\eqref{prim} not involving the hydrodynamic pressure $p$ is the key feature that allows us to prove that such solutions do not exist. For the case of the 3D Euler equations one cannot eliminate the pressure from the equations, which explains the crucial role played by $p$ in Gavrilov's construction~\cite{Gavrilov}. For example, it is easy to check that there cannot exist a $C^\infty$ steady Euler flow $X$ on $\RR^3$ with compact support and everywhere constant pressure (because, under these assumptions, $|X|^2$ would be a first integral and the integral curves of $X$ would be straight lines).
\end{re}

\subsection{Step~2: The case $R=0$, $\nu_H = \nu_3 =0$}



A straightforward application of the vector calculus identity for the  divergence of a vector field multiplied by a scalar function, implies that the divergence-free vector field $X$ satisfies the identities
\begin{equation}\label{idxy}
x X\cdot \nabla u = \di(x u X)-u^2\,, \qquad
y X\cdot \nabla v = \di(y v X)-v^2\,.
\end{equation}
Since, by Step~1, $X$ solves the reduced equations~\eqref{prim}, these identities become
\begin{equation}\label{idxyprim}
R x v = \di(x u X)-u^2\,, \qquad -R y u = \di(y v X)-v^2\,,
\end{equation}
and their integration on $\Omega$ yields (recall that $X$ has compact support):
\begin{equation}\label{idxyprimint}
\int_\Omega (R x v + u^2)\dif \mu =0\, , \qquad
\int_\Omega (-R y u + v^2)\dif \mu =0\,.
\end{equation}
When $R=0$, Equations~\eqref{idxyprimint} imply that $u\equiv 0$ and $v\equiv 0$ on $\Omega$, and hence the solenoidal vector field $X$ is of the form $X=w(x,y) e_3$ for some smooth function $w$. In turn, such a vector field cannot be compactly supported unless $w\equiv 0$, which completes the proof of the theorem in this case.

\subsection{Step~3: the case $R\neq 0$, $\nu_H = \nu_3 =0$}

We recall that a first integral of a vector field $X$ is a $C^1$ function that is constant along the stream lines of the flow of $X$. The following simple observation will be crucial in what follows.

\begin{lm}\label{L.firstint}
The smooth functions
\begin{equation*}
I_1=u-Ry\,,\qquad I_2=v+Rx\,,
\end{equation*}
are first integrals of any vector field $X=ue_1+ve_2+we_3$ that satisfies the reduced equations~\eqref{prim} in the invicid case $\nu_H = \nu_3 =0$.
\end{lm}
\begin{proof}
Since $X \cdot \nabla x = u$ and $X \cdot \nabla y = v$, the first two equations in \eqref{prim} readily rewrite as $X \cdot \nabla I_1=0$ and $X \cdot \nabla I_2=0$, which establish the lemma.
\end{proof}

Let us consider the $C^\infty$ map $I:= (I_1,I_2): \Omega \to \RR^2$; by Lemma~\ref{L.firstint}, $X \in \ker \dif I$. We use the notation
$$C_I:= \{p \in\Omega: \mathrm{rank}(\dif I)_p < 2\}$$
for the critical set of $I$ (which is a closed set). Since $I=(-Ry,Rx)$ on $\Omega\backslash \text{int}(K)$, we infer that $\mathrm{rank}(\dif I) = 2$ on the closure of the complement of $K$, and therefore
$$C_I\Subset K\,,$$
i.e., it is a proper subset of the support of $X$. Then, for any point $q \in \partial K:=K\backslash \text{int}(K)$, there is a small constant $\epsilon>0$ such that the rank of $(\dif I)_p$ is $2$ at any point $p$ in the ball $B(q, \epsilon)$ of radius $\epsilon$ centred at $q$. Moreover, by the Local Submersion Theorem, the fibres of the submersion $I: B(q, \epsilon) \to I(B(q, \epsilon))$ are connected lines, provided that $\epsilon$ is small enough. For the rest of the proof, we choose $q\in\partial K$ such that $\text{dist}(q,\partial\Omega)=\text{dist}(K,\partial\Omega)$, i.e., a point on $\partial K$ that is as close as possible to $\partial\Omega$. By compactness, this obviously exists (it does not need to be unique). In particular, no piece of a vertical line can be contained in $\partial K\cap B(q,\epsilon)$, provided that $\epsilon$ is small enough.

The fact that $\dif I$ has maximal rank on $B(q, \epsilon)$ implies that the vector field $\nabla I_1\times \nabla I_2$ does not vanish at any point of this neighborhood. Then, defining the function
\[
F:=\frac{X\cdot (\nabla I_1\times\nabla I_2)}{|\nabla I_1\times \nabla I_2|^2}\,,
\]
which is clearly of class $C^\infty$ in $B(q, \epsilon)$, we easily infer that
\begin{equation}\label{Fdef}
X = F \, \nabla I_1\times \nabla I_2
\end{equation}
in $B(q, \epsilon)$. The condition $\di X=0$ together with the obvious fact that $\di (\nabla I_1\times \nabla I_2) =0$, imply that $F$ is a \emph{basic function}, i.e.,
\[
\nabla F\cdot (\nabla I_1\times \nabla I_2)=0\,.
\]
Since the fibres of the submersion $I$ are connected on $B(q,\epsilon)$, $F$ being basic is equivalent~\cite[Chapter 2.1]{mol} to the existence of a function $\ov F\in C^\infty(\RR^2)$ such that
\begin{equation}\label{Fproj}
F=\ov F \circ I\,.
\end{equation}

Now, we set $I_0$ to be the submersion $I_0(x,y,z):=(-Ry,Rx)$. Clearly the fibres of $I_0$ are the vertical lines and
\[I=I_0+(u,v)\,.\]

At this point the idea of the proof goes as follows. We show that $I$ can be seen as a perturbation of $I_0$ that is ``small'' with respect to the $C^k$ norm. This allows us to apply Thom's isotopy theorem to infer that the fibres of $I$ are ``close'' to the vertical lines and, as a consequence, connect the points outside $K$ with points in the interior of $K$. But $F$ is constant along these fibres by \eqref{Fproj}, so it will have to vanish, together with $X$, in a small open set inside the support $K$. This will yield the desired contradiction.

\begin{lm}
Fix an integer $k\geq 1$. For any  $N \in \NN$ we have
\begin{equation}\label{eq.esti}
\norm{I - I_0}_{C^k(B(q, \epsilon))} < C_{N,k} \epsilon^N
\end{equation}
for some $\epsilon$-independent constant $C_{N,k}:= \max_{\abs{\beta}=N}\|\partial^\beta X\|_{C^{k}(\Omega)}$.
\end{lm}

\begin{proof} 
First notice that for any multi-index $\alpha \in \NN^3$, we have $\partial^\alpha(I-I_0)(q)=0$ because $q\in\partial K$ and the smooth function $I - I_0$ is identically zero on $\Omega \setminus \mathrm{int}(K)$. Then, for each $|\alpha|\leq k$, Taylor's theorem allows us to write the estimate
$$\abs{\partial^\alpha (I-I_0)(p)} \leqslant  c_{N,\abs{\alpha}}\epsilon^N$$
for any $p \in B(q, \epsilon)$, with $\displaystyle c_{N,\abs{\alpha}}:= \frac{1}{N!}\max_{\abs{\beta} = N + \abs{\alpha}}\sup_{p \in B(q, \epsilon)}\abs{\partial^\beta (I-I_0)(p)}$. We can conclude that
$$\|I - I_0\|_{C^k(B(q, \epsilon))} := \max_{\abs{\alpha}\leq k}\sup_{p \in B(q, \epsilon)} \abs{\partial^\alpha (I-I_0)(p)} < C_{N,k} \epsilon^N\, ,$$
with $C_{N,k}:= \max_{\abs{\beta}=N}\|\partial^\beta X\|_{C^{k}(\Omega)}$.
\end{proof}

We will see next how the relation between $I$ and $I_0$ can be mirrored by associated submersions defined on the unit ball $B$ centered at the origin. This is important for the application of Thom's isotopy theorem.

Let $\Lambda_\epsilon: B(q, \epsilon) \to B$, $\Lambda_\epsilon(x,y,z)=(\bar x,\bar y,\bar z)$ be the diffeomorphism given by the affine change of coordinates:
\[
(x,y,z)=q+\epsilon(\bar x,\bar y,\bar z)\,.
\]
Consider $\bar I=I \circ \Lambda_\epsilon^{-1}$ and $\bar I_0 = I_0 \circ \Lambda_\epsilon^{-1}$ that define submersions on $B$. Explicitly, we have $\bar I_0(\bar x,\bar y,\bar z)=(-Rq_y,Rq_x)+\epsilon(-R\bar y,R\bar x)$, with $q=:(q_x,q_y,q_z)$.

Fixing an integer $N_0\geq 1$, the estimate~\eqref{eq.esti} and the definition of the constant $C_{N,k}$ implies
\[
\|\bar I - \bar I_0\|_{C^k(B)}\leq C\epsilon^{N_0 + 1}
\]
for some $\epsilon$-independent constant $C>0$ (the extra power of $\epsilon$ is a straightforward consequence of the chain rule and the definition of the diffeomorphism $\Lambda_\epsilon$). Accordingly, if we define
\[
\hat I:=\frac{1}{\epsilon}\left(\bar I - (-Rq_y,Rq_x)\right)\,, \quad
\hat I_0(\bar x,\bar y,\bar z) := (-R\bar y,R\bar x)
\]
which are $C^\infty$ submersions on $B$, we obtain
\[
\hat I = \hat I_0 + \mathcal O_k(\epsilon^{N_0})\,.
\]
Here the notation $\mathcal O_k(\epsilon^{N_0})$ means that the $C^k$ norm of this term is bounded as $C\epsilon^{N_0}$. Obviously, the fibres of $\hat I$ and $\bar I$ coincide up to a change of values.

Since $N_0\geq 1$, $\hat I$ is a $C^k$-small ($k\geq 1$) perturbation of the submersion $\hat I_0$. Thom's isotopy theorem~\cite[Theorem 20.2]{abra} (see also~\cite[Theorem 3.1]{EP13}) then implies that for each fibre $\Gamma_c:=\hat I^{-1}(c)$ there is a smooth diffeomorphism $\Phi_c$ of $B$ such that $\Phi_c(\Gamma_c)$ is the fiber $\hat I_0^{-1}(c)$, which is a vertical line. Moreover, the diffeomorphism is close to the identity as
\begin{equation}\label{eq.diff}
\|\Phi_c-id\|_{C^k(B)}< C'\epsilon^{N_0}
\end{equation}
for some $\epsilon$ and $c$-independent constant $C'$ (the facts that we can do the above estimate with $\epsilon^{N_0}$ and that, by construction, $C'$ is $\epsilon$ and $c$-independent, are consequences of the last part of the proof of \cite[Theorem 3.1]{EP13}).

Recall now that no piece of a vertical line can be contained in $\partial K\cap B(q,\epsilon)$ (by the choice of the point $q$). Of course, the same holds true in $B$, that is, the corresponding boundary set $\partial\Lambda_\epsilon(K) \cap B$ does not contain a vertical line, and hence the fibres of $\hat I_0$ cross this boundary set. Since ``crossing the boundary'' is an open condition, it is preserved by the small perturbation in Equation~\eqref{eq.diff}. More precisely, consider a fibre $\hat I_0^{-1}(c)$, and two points $p_1$ and $p_2$ on this fibre so that $p_1\in \Lambda_\epsilon(K)\cap B$ and $p_2\in B\backslash \overline{\Lambda_\epsilon(K)\cap B}$. Since these sets are open, the corresponding $c$-fibre of $\hat I$ contains points $\Phi_c^{-1}(p_1),\Phi_c^{-1}(p_2)$ that lie in the same sets $\Lambda_\epsilon(K)\cap B$ and $B\backslash \overline{\Lambda_\epsilon(K)\cap B}$, respectively, provided that $\epsilon$ is small enough (cf. the estimate~\eqref{eq.diff}). Applying this reasoning to all the values $c$ in the compact set $\overline{\hat I_0(B_{1/2})}$ ($B_{1/2}$ being the closed ball of radius $1/2$, or any other radius in the interval $(0,1)$), we conclude that all the fibres of the submersion $\hat I$ that intersect $B_{1/2}$ cross the boundary $\Lambda_\epsilon(\partial K)$. In turn, this implies that the fibres of $I$ intersecting $B(q,\epsilon/2)$ cross $\partial K$.

Finally, since $X$ vanishes in the complement of $K$, the function $F$ also vanishes in the set $B(q,\epsilon)\backslash K$. But we have established before that all the fibres that intersect the ball $B(q,\epsilon/2)$ also intersect $B(q,\epsilon)\backslash K$, so being $F$ a function of the form $\ov F \circ I$, we conclude that $F=0$ on the whole $B(q,\epsilon/2)$. Accordingly, $X=0$ on $B(q, \epsilon/2) \cap \mathrm{int}(K)$ and this is a contradiction with the fact that $K$ is the support of $X$, unless $X\equiv 0$ on $\Omega$. This completes the proof of the theorem in the case that $R\neq 0$.

\subsection{Step~4: the case $\nu_H + \nu_3 >0$}

This case is elementary. If we multiply the first equation in \eqref{prim} by $u$ and integrate on $\Omega$, and multiply the second equation in \eqref{prim} by $v$ and integrate on $\Omega$, the sum of both equations yields the identity
\[
\nu_H\int_\Omega\Big(|\nabla_H u|^2+|\nabla_H v|^2\Big)\dif \mu + \nu_3\int_\Omega\Big(u_z^2+v_z^2\Big)\dif \mu = 0\, ,
\]
where $\nabla_H$ is the gradient with respect to the $(x,y)$ variables.

If $\nu_H$ and $\nu_3$ are both positive, then the above identity implies (independently of the value of $R$) that $u\equiv 0$ and $v\equiv 0$ on $\Omega$. The solenoidal vector field $X$ then is of the form $w(x,y)e_3$, which cannot be compactly supported unless $w\equiv 0$ on $\Omega$.

If $\nu_H >0$ and $\nu_3=0$, then the above identity implies that $\nabla_H u \equiv 0$ and $\nabla_H v \equiv 0$ on $\Omega$, and hence $u$ and $v$ are functions that depend only on the variable $z$. This is compatible with having compact support if and only if $u\equiv 0$ and $v\equiv 0$ on $\Omega$. Arguing as before, we conclude that $X\equiv 0$.

Finally, if $\nu_H =0$ and $\nu_3>0$, then the above integral identity implies that $\partial_z u \equiv 0$ and $\partial_z v \equiv 0$ on $\Omega$, which implies that $X$ must be of the form $u(x,y)e_1+v(x,y)e_2 + w(x,y)e_3$. Again, $X$ cannot be compactly supported unless it is identically zero on the whole $\Omega$.

Since by definition $\nu_H\geqslant 0$ and $\nu_3\geqslant 0$, this completes the proof of the viscous case.

\begin{re}
With simple adaptations, the same proof works if $X$ is assumed to be a compactly supported solution of class $C^k$, with $k\geq 2$.
\end{re}


\section{Final remark: axisymmetric solutions}\label{S:final}

In view of Gavrilov's compactly supported solutions to the stationary Euler equations, it is useful to analyze the particular case of axisymmetric solutions to the stationary primitive equations. So we assume throughout this section that $\nu_H=\nu_3=0$, since the case with viscosity turned out to be much simpler to analyze (cf. Step~4, Section~\ref{S:proof}).

In Step~1 of Section~\ref{S:proof} we showed that any stationary solution $X$ to the primitive equations with compact support in $\Omega$ satisfies the reduced equations~\eqref{prim}. In terms of cylindrical coordinates $(r, \theta, z)\in \mathbb R^+\times \mathbb S^1\times (0,L)$, if the divergence-free vector field $X$ is assumed to be axisymmetric, it reads as
\begin{equation}\label{axi}
X=\frac{1}{r}\Big(-\partial_z \psi(r,z)e_r+\partial_r \psi(r,z)e_z\Big) + S(r,z)e_\theta\,,
\end{equation}
for some stream function $\psi$ and swirl $S$ that depend only on $(r,z)$. Here $\{e_r,e_\theta,e_z\}$ is the cylindrical orthonormal framing.

Substituting the expression~\eqref{axi} into the equations~\eqref{prim} we easily obtain the equations:
\begin{equation}\label{jac}
\partial_{r}\psi \, \partial_{z}(r S)- \partial_z\psi \, \partial_r \! \left(r S + \tfrac{R}{2}r^2\right)=0\,,
\end{equation}
and
\begin{equation}\label{2nd_axi_eq}
\partial_{zz}\psi \, \partial_r \psi-\partial_{rz}\psi \, \partial_z\psi + \tfrac{1}{r}(\partial_z\psi)^2 + r S(S+Rr) = 0\,.
\end{equation}

Equation~\eqref{jac} means that the Jacobian of the couple of functions $(\psi, \ r S + \frac{R}{2}r^2)$, as a function of $(r,z)$, is zero. This leads us to the following ansatz for the swirl function
\begin{equation}\label{swirl}
S(r,z) = \frac{F(\psi(r,z))}{r}-\frac{Rr}{2}\,,
\end{equation}
which exactly coincides, if $R=0$, with the one employed in the study of axisymmetric Euler flows~\cite{CV,DEP}.

\begin{re}
Any function $S(r,z)$ of the form presented in Equation~\eqref{swirl} satisfies Equation~\eqref{jac}, but the converse implication is not true in general. As a consequence of a theorem proved in~\cite{Newns}, the best we can say is that there is a smooth function $G\in C^\infty(\mathbb R^2)$ such that
\[
G\Big(\psi(r,z),\ rS(r,z)+\frac{R}{2}r^2\Big)=0\,.
\]
\end{re}

Replacing the ansatz~\eqref{swirl} in Equation \eqref{2nd_axi_eq} we obtain:
\begin{equation}\label{GSprimeq}
\partial_{zz}\psi \, \partial_r \psi-\partial_{rz}\psi \, \partial_z\psi + \tfrac{1}{r}\left(F(\psi)^2+(\partial_z\psi)^2\right)=\frac{R^2}{4}r^3\, ,
\end{equation}
that plays the role of the \textit{Grad-Shafranov equation} in this case. When $R\neq 0$, it is obvious that there cannot be compactly supported solutions. Indeed, denoting by $K$ the compact set where $X$ is supported, $\psi$ would be a constant $c$ on $\Omega\backslash K$, but then the swirl function $S$ given by Equation~\eqref{swirl} is of the form
\[
S=\frac{F(c)}{r}-\frac{Rr}{2}\,,
\]
on $\Omega\backslash K$, which cannot be identically zero.

When $R=0$ we can integrate the Grad-Shafranov type equation~\eqref{GSprimeq} on $\Omega$. Using that
\begin{align*}
\int_0^L \!\! \int_0^\infty(\partial_{zz}\psi \, \partial_r \psi-\partial_{rz}\psi \, \partial_z\psi)rdrdz = -\int_0^L \!\! \int_0^\infty \partial_r\big((\partial_z\psi)^2\big)rdrdz =
\int_0^L \!\! \int_0^\infty(\partial_z\psi)^2drdz\,,
\end{align*}
by integration by parts, we find that
$$\int_0^L \!\!\! \int_{0}^\infty \Big(2(\partial_z \psi)^2 + F(\psi)^2\Big) drdz =0\,.$$
We then conclude that $X$ is of the form
\[
X=\frac{\psi'(r)}{r}e_z\,,
\]
which cannot be compactly supported in $\Omega$ unless $X\equiv 0$.


%

\section*{Acknowledgments}
The authors are very grateful to Sergei Kuksin for proposing this problem. We also thank Edriss S. Titi for providing references and interesting comments.  D.P.-S. is supported by the grants CEX2019-000904-S and PID2019-106715GB GB-C21 funded by MCIN/AEI/10.13039/501100011033, and also acknowledges partial support from the grant ``Computational, dynamical and geometrical complexity in fluid dynamics'', Ayudas Fundaci\'on BBVA a Proyectos de Investigaci\'on Cient\'ifica 2021.

\end{document}